\documentclass{amsart}

\usepackage{amssymb, amsmath, amsthm, hyperref}

\usepackage[english]{babel}

\author{Rafael Torres}

\title[Irreducible nonspin 4-manifolds with abelian fundamental group]{Geography and botany of irreducible nonspin symplectic 4-manifolds with abelian fundamental group}



\address{California Institute of Technology - Mathematics\\ 1200 E California Blvd\\91125\\Pasadena, CA}

\email{rtorresr@caltech.edu}

\theoremstyle{plain}

\newtheorem{theorem}[equation]{Theorem}

\newtheorem{corollary}[equation]{Corollary}

\newtheorem{proposition}[equation]{Proposition}

\newtheorem{lemma}[equation]{Lemma}

\newtheorem{remark}[equation]{Remark}

\theoremstyle{definition}

\newtheorem{definition}[equation]{Definition}

\newtheorem{example}[equation]{Example}

\newcommand{\R}{\mathbb{R}}

\newcommand{\Z}{\mathbb{Z}}

\newcommand{\N}{\mathbb{N}}

\begin{document}

\maketitle

\emph{Abstract:} The geography and botany problems of irreducible nonspin symplectic 4-manifolds with a choice of fundamental group from $\{\Z_p, \Z_p\oplus \Z_q, \Z, \Z\oplus \Z_p, \Z\oplus \Z\}$ are studied by building upon the recent progress obtained on the simply connected realm. Results on the botany of simply connected 4-manifolds not available in the literature are extended.

\section{Introduction}

Topologists' understanding of smooth 4-manifolds has witnessed a drastic improvement during the last twenty years. Advances in symplectic geometry \cite{[T1], [Lu], [T0], [Gom], [T2], [ADK], [Ush], [HKo]} and inventions of gluing formulas \cite{[MMS], [BPa3]} for diffeomorphism invariants \cite{[W]} have paired elegantly with topological constructions \cite{[FS3], [FS6], [FS], [FPS], [NFS]}, offering a noteworthy insight into the smooth four-dimensional story. The most recent series of successes  \cite{[AP], [AP1], [BK2], [FPS], [ABBKP], [ABP], [AP], [NFS]} has answered as many questions as it has raised new ones, and 4-manifold theory remains to be an intriguing and active area of research.\\

Two factors responsible for the recent progress on the simply connected realm are the increase in the repertoire of techniques that manufacture symplectic 4-manifolds with small topological invariants, and a new perspective on the usage of already existing mechanisms of construction. The idea of
using symplectic sums \cite{[Gom]} of non-simply connected building blocks along genus 2 surfaces to kill fundamental groups in an
efficient way was introduced in \cite{[A1]}. Its immediate outcome was the construction of an exotic symplectic $\mathbb{CP}^2\#5\overline{\mathbb{CP}^2}$ and, later on, the existence of an exotic symplectic $\mathbb{CP}^2 \# 3\overline{\mathbb{CP}^2}$ 
\cite{[AP1]} was put on display. Shortly after, Luttinger surgery \cite{[Lu], [ADK]} was introduced to the list of symplectic constructions in
\cite{[BK2], [Stt]}. The combinations of these 
techniques yielded another construction of an exotic symplectic $\mathbb{CP}^2 \# 3\overline{\mathbb{CP}^2}$ in \cite{[FPS], [BK2]}. Eventually these techniques were succesful in unveiling exotic smooth structures on $\mathbb{CP}^2 \# 2 \overline{\mathbb{CP}^2}$ \cite{[AP], [NFS]}, the 4-manifold with the smallest Euler characterstic known to admit exotic smooth structures at the time this article was written.\\

Having in mind as a motivation Freedman's topological classification of simply connected 4-manifolds \cite{[F]}, the knowledge that has been accumulated on smooth 4-manifolds can be encoded to address the questions of existence and uniqueness on a possible classification scheme as follows. The \emph{geography problem} for symplectic 4-manifolds with a given fundamental group \cite{[MW], [Gom]} asks which homeomorphism classes are realized by an irreducible symplectic 4-manifold. The \emph{botany problem} \cite{[GS]} asks how many diffeomorphism classes exist within a given homeomorphism type. The geography problem for symplectic manifolds with trivial fundamental group was first studied systematically in \cite{[JPG]}, where regions are populated using the techniques of \cite{[Gom]}. The novel tools recently introduced to unveil a myriad of smooth structures \cite{[FS3], [FS], [FPS]} have allowed a hand-in-hand study of both problems. As a sample of the noteworthy successes in the area, it is worth pointing out that the symplectic geography question has been settled for simply connected 4-manifolds of negative signature, and it is known that these manifolds admit infinitely many smooth structures \cite{[Gom], [BPa], [JP], [FPS], [ABBKP], [AP]}.\\

The nonsimply connected realm remains to be somewhat uncharted territory. The topological classification of closed oriented 4-manifolds for several choices of good fundamental groups \cite{[FQ]}  have been established in \cite{[HK90], [HK93], [HT], [KL]}. These considerations motivate curiosity minding the symplectic geography and botany for manifolds with nontrivial fundamental group.\\

The contribution and purpose of this paper is to show how current efforts set on the simply connected realm extend naturally to smooth 4-manifolds with abelian fundamental groups of small rank, amongst other choices of fundamental groups. This has been pointed out previously by other authors \cite{[Gom], [BK3], [ABBKP], [AP]}, and extends work of the author done in \cite{[To], [To1]}. We build greatly upon the recent constructions to systematically study the geography and botany of irreducible symplectic 4-manifolds with abelian fundamental groups. The first examples of exotic 4-manifolds with cyclic fundamental group were constructed in \cite{[LO], [FM], [HK88], [OKo], [HK90], [Wa], [FST], [Gom], [BK3], [AP1]}, and examples of 4-manifolds with finite fundamental group can be found in \cite{[HK90], [Gom]}. Examples of minimal symplectic 4-manifolds with fundamental group $\Z \oplus \Z_p$ had been previously built in \cite{[Gom], [OS], [IS]}. Efforts towards more general fundamental groups can be found in \cite{[Gom], [BK1], [JPa], [BK3], [BK4], [Y]}.\\

The organization of the paper is as follows. The main result and the notation employed are presented in Section \ref{Section 2}. The results on homeomorphism criterions are recalled in Section \ref{Section 4.1}, Section \ref{Section 4.2}, and Section \ref{Section 4.3}. The fundamental building blocks are presented in Section \ref{Section 4.6} and Section \ref{Section 4.9}, and the basic construction tool in Section \ref{Section 4.5}. Our two main technical results are presented in Section \ref{Section 4.8}. With the purpose of providing the reader with hands-on constructions, two examples are worked out thoroughly in Section \ref{Section 4.7}. Finally, in Section \ref{Section 5} the efforts are put together to fill in regions of the symplectic geography for the choices of fundamental groups involved, and address the botany. The region of 4-manifolds with negative signature is populated in Section \ref{Section 5.2}, and a region of nonnegative signature in Section \ref{Section 5.3}.

\section{Notation and main results.}{\label{Section 2}}

A finite cyclic group of order $p\geq 2$ is denoted by $\Z_p$. The homeomorphism types of the 4-manifolds with nonspin universal cover, and whose  fundamental group is among the choices considered are given by 
\begin{center}\begin{itemize}
\item $\pi_1 = \Z_p\oplus \Z_q: b^+_2\mathbb{CP}^2 \# b^-_2\overline{\mathbb{CP}^2} \# L_{p, q}$,

\item $\pi_1 = \Z_p: b^+_2\mathbb{CP}^2 \# b^-_2\overline{\mathbb{CP}^2} \#L_p$, and 

\item $\pi_1 = \Z: b^+_2\mathbb{CP}^2 \# b^-_2\overline{\mathbb{CP}^2} \# S^1\times S^3$
\end{itemize}
\end{center}

according to \cite{[FQ], [HK93], [HT], [KL]} (see Section \ref{Section 4.1}). The pieces $L_{p, q}$ and $L_p$ are oriented closed smooth 4-manifolds that satisfy $\pi_1(L_{p, q}) = \Z_p \oplus \Z_q$, $\pi_1(L_p) = \Z_p$, $e(L_{p, q}) = 2 = e(L_p)$, and $\sigma(L_{p, q}) = 0 = \sigma(L_p)$. These manifolds are constructed as follows.  Consider the product $L(p, 1)\times S^1$ of a Lens space and a circle, and the map

\begin{center}
$\phi: L(p, 1)\times S^1 \rightarrow L(p, 1)\times S^1$

$\{pt\} \times \alpha \mapsto \{pt\} \times \alpha^q$.
\end{center}

Carve out the loop $\alpha^q$ and glue in a copy of $S^2\times D^2$ in order to kill the
generator that corresponds to the $\Z$-factor in $\pi_1(L(p, 1)\times S^1) \cong \Z_p \oplus \Z$ to obtain a manifold
\begin{center}
$L_{p, q}:= (L(p, 1)\times S^1 - (S^1\times D^3)) \cup_{\phi} (S^2\times D^2)$.
\end{center}

Using the Seifert-van Kampen theorem, one concludes $\pi_1(L_{p, q}) = \Z_p \oplus \Z_q$. The case $q = 1$ yields the manifold $L_p:= L_{p, 1}$ with finite cyclic fundamental group of order $p$.\\

A smooth closed 4-manifold $X$ is \emph{irreducible} if for every smooth connected sum decomposition $X = X_1 \# X_2$, either $X_1$ or $X_2$ is homeomorphic to $S^4$ \cite[Definition 10.1.17]{[GS]}. It is called \emph{minimal} if it is not the blow-up of another manifold \cite[p. 46]{[GS]}. The following definition will be used for our purposes (cf. \cite[Definition 1]{[AP2]}).\\

\begin{definition}{\label{Definition 1}} The homeomorphism class of a 4-manifold $X$ has the \emph{$\infty$-property} if and only if there
exists an infinite set $\{X_n: n\in \mathbb{N}\}$ that consists of minimal 4-manifolds such that

\begin{itemize}
\item if $n\neq m$, then $X_n$ is nondiffeomorphic to $X_m$
\item there exists an $m\in \N$ for which there is an element $X_m \in \{X_n\}$ that is symplectic, while $X_m$ is nonsymplectic for  every $m\neq n$; we will assume $m = 1$.
\item $X_n$ is homeomorphic to $X$ for all $n\in \N$.
\end{itemize}
The homeomorphism class of a 4-manifold $X$ has the \emph{$\infty^2$-property} if and only if there exists an infinite set $\{X_n : n\in \mathbb{N}\}$ that contains infinitely many irreducible symplectic and infinitely many minimal nonsymplectic 4-manifolds such that
\begin{itemize}
\item if $n\neq m$, then $X_n$ is nondiffeomorphic to $X_m$
\item $X_n$ is homeomorphic to $X$ for all $n\in \N$.\\
\end{itemize}
\end{definition}


The main result of the paper is the following theorem.

\begin{theorem}{\label{Theorem 2}}  Let $G \in \{\Z_p, \Z_p \oplus \Z_q, \Z \oplus \Z_q, \Z, \Z \oplus \Z\}$. Let $(c, \chi)$ be any pair of non-negative integers $($except for $(1, 1)$$)$, which satisfies
\begin{center}
$0\leq c \leq 8\chi - 1$.
\end{center}
There exists an irreducible symplectic 4-manifold $X^G$ such that 
\begin{center}
$\pi_1(X^G) = G$, $(c_1^2(X^G), \chi_h(X^G)) = (c, \chi)$.
\end{center}

If $\chi = 1$, then the homeomorphism types

\begin{itemize}
\item \begin{center} $\mathbb{CP}^2 \# (10 - c - 1)\overline{\mathbb{CP}^2}\# L_p$
\end{center}
\end{itemize}

have the $\infty$-property.\\

If $\chi \geq 2$ and $q$ is an odd prime number, then the homeomorphism types

\begin{itemize}
\item \begin{center} $(2\chi - 1)\mathbb{CP}^2 \# (10\chi - c - 1)\overline{\mathbb{CP}^2}$,
\end{center}
\item \begin{center}
$(2\chi - 1) \mathbb{CP}^2 \# (10\chi - c - 1)\overline{\mathbb{CP}^2} \# 
L_p$,
\end{center}
\item \begin{center} $(2\chi - 1) \mathbb{CP}^2 \# (10\chi - c - 1)\overline{\mathbb{CP}^2} \# 
L_{q, q}$, and 
\end{center}

\item \begin{center} $2\chi \mathbb{CP}^2 \# (10\chi - c)\overline{\mathbb{CP}^2} \# S^1\times S^3$
\end{center}
\end{itemize}
have the $\infty^2$-property.
\end{theorem}

\section{Background results, construction tools and building blocks.}{\label{Section 4}}

\subsection{Homeomorphism type of manifolds with $\pi_1 = \Z_p$.}{\label{Section 4.1}}

The topological classification of closed oriented 4-manifolds with finite cyclic group was obtained by Hambleton and Kreck \cite{[HK93]}. Due to the existence of 2-torsion, one must be careful when determining the parity of its intersection form. Let $\widetilde{X}\rightarrow X$ be the universal cover of the nonsimply connected 4-manifold $X$.

\begin{itemize}
\item $\omega_2$-type (I) if $\omega_2(\widetilde{X}) \neq 0$,
\item $\omega_2$-type (II) if $\omega_2(X) = 0$,
\item $\omega_2$-type (III) if $\omega_2(X) \neq 0$, but $\omega_2(\widetilde{X}) = 0$.
\end{itemize}

\begin{theorem}{\label{Theorem 6}} (Hambleton - Kreck, \cite[Theorem C]{[HK93]}). A smooth closed oriented 4-manifold $X$ with finite cyclic fundamental group is classified up to homeomorphism by the fundamental group, the intersection form on $H_2(X; \Z)$/Tors, and the $\omega_2$-type. Moreover, any isometry of the intersection form can be realized by a homeomorphism.
\end{theorem}

By using the known work of Donaldson and of Minkowski-Hasse on the classification of bilinear forms over the integers that are realized as intersection forms of smooth 4-manifolds, the previous result can be restated as follows.

\begin{theorem}{\label{Theorem 7}} A smooth, closed, oriented 4-manifold with finite cyclic fundamental  and indefinite intersection form is classified up to homeomorphism by the fundamental
group, the Betti numbers $b^+_2$ and $b^-_2$, the parity of the intersection form and the $\omega_2$-type.
\end{theorem}

\subsection{Homeomorphism type of manifolds with $\pi_1 = \Z_q\oplus \Z_q$, $q$ an odd prime number.}{\label{Section 4.2}}

The classification up to homeomorphism of 4-manifolds with fundamental group among the choices of finite noncyclic abelian groups under consideration is given in the following theorem.

\begin{theorem}{\label{Theorem 8}} (Hambleton - Kreck, \cite[Theorem B]{[HK93]}). Let $X$ be a smooth closed oriented 4-manifold, and let $\pi_1(X) = \pi$ be a finite group of
odd order. When $\omega_2(\tilde{X}) = 0$ (resp. $\omega_2(\tilde{X}) \neq 0$), assume that
\begin{center}
$b_2(X) - |\sigma(X)| > 2 d(\pi)$,
\end{center}
$(resp.  >2d(\pi) + 2)$. Then, $X$ is classified up to homeomorphism by the signature, Euler characteristic, type, and fundamental class in $H_4(\pi, \Z) / Out(\pi)$.
\end{theorem}

In the statement of Theorem \ref{Theorem 8}, the outer automorphism group of $\pi$ is denoted by $Out(\pi)$. The stability condition required on the lower bound for the Euler 
characteristic, $d(\pi) \in \Z$ (see \cite{[HK93]} for details), depends on the fundamental group of the manifold. In \cite{[To]}, it was proven that $d(\Z_q \oplus \Z_q) = 1$.

\subsection{Homeomorphism type of manifolds with $\pi_1 = \Z$.}{\label{Section 4.3}}

Provided a stability condition holds, the homeomorphism criterion for oriented smooth 4-manifolds with $\pi_1 = \Z$ is similar to Freedman's known result \cite{[F]} (see also \cite{[FQ], [SW], [KL]}).

\begin{theorem} {\label{Theorem 9}} (Hambleton - Teichner, \cite[Corollary 3]{[HT]}). If $X$ is a closed oriented smooth 4-manifold with infinite cyclic fundamental group and satisfies the inequality
\begin{center}
$b_2(X) - |\sigma(X)| \geq 6$,
\end{center}
then $X$ is homeomorphic to the connected sum of $S^1 \times S^3$ with a unique closed
simply-connected 4-manifold. In particular, $X$ is determined up to homeomorphism by its second Betti
number $b_2(X)$, its signature $\sigma(X)$ and its $\omega_2$-type. Moreover, $X$ is either spin or nonspin depending on the parity of its intersection form.
\end{theorem}

\subsection{Basic symplectic building blocks.} {\label{Section 4.6}}

Consider the 4-torus $T^4 = T^2 \times T^2$ equipped with the product symplectic form. Let $x, y, a, b$ denote both the generators of the group $\pi_1(T^2\times T^2) = \Z x \oplus \Z y \oplus \Z a \oplus \Z b$, and the corresponding loops as well. This convention will be maintained through out the paper. The tori \begin{center} $T_1:= x \times a, T_2:= y\times a$ \end{center} and their respective geometrically dual tori $T_1^d:= y\times b$,  $T_2^d:= x\times b$ are Lagrangian, and the torus $T_3:= a\times b$ and its geometrically dual torus $T_3^d:= x\times y$ are symplectic. The characteristic numbers are $c_1^2(T^4) = 0 = \chi_h(T^4)$.\\

The calculation of the fundamental group of the complement of surfaces inside a manifold plays a fundamental role in recent constructions of 4-manifolds. Choices of basepoints are fundamental in the application of the Seifert-van Kampen theorem, which is to be used with care. For these matters we build greatly in the analysis done in \cite{[BK2], [BK4], [BK3]}. In particular, we have the following result.

\begin{proposition}{\label{Proposition 14}} (Baldridge - Kirk \cite[Theorem 2]{[BK2]}). The fundamental group \begin{center} $\pi_1(T^4 - (T_1 \cup T_2))$\end{center} is generated by the loops $x, y, a, b$ and the relations $[x, a] = [y, a] = 1$ hold. The meridians of the tori and the two Lagrangian push offs of their generators are given by:
\begin{center}
$T_1: \mu_1 = [b^{-1}, y^{-1}], m_1 = x, l_1 = a$, and\\
$T_2: \mu_2 = [x^{-1}, b], m_2 = y, l_2 = bab^{-1}$.\\
\end{center}
\end{proposition}

As our next building block, we take the product of a genus two surface and a torus, $\Sigma_2 \times T^2$, and endow it with the product symplectic form. Its characteristic numbers are $c_1^2(\Sigma_2 \times T^2) = 0 = \chi_h(\Sigma_2 \times T^2)$. Let $a_1, b_1, a_2, b_2$ be the loops and generators of $\pi_1(\Sigma_2)$, and $x, y$ the loops and generators of $\pi_1(T^2)$. Inside this manifold there are four pairs of homologically essential Lagrangian tori, and a symplectic surface of genus two and self intersection zero. The tori are displayed in the statement of the following proposition; the genus two surface is a parallel copy of the surface $\Sigma_2 \times \{pt\}$, and we will denote it by $\Sigma_2$.  

\begin{proposition}{\label{Proposition 15}} (Baldridge - Kirk  \cite[Proposition 7]{[BK3]}). The fundamental group
\begin{center}$\pi_1(\Sigma_2\times T^2 - (\Sigma_2 \cup T_1 \cup \cdots \cup T_4))$\end{center}
is generated by the loops  $x, y, a_1, b_1, a_2, b_2$. Moreover,  with respect to certain paths to the boundary of the tubular neighborhoods of the $T_i$ and $\Sigma_2$, the meridians and two Lagrangian push offs of the surfaces are given by
\begin{itemize}
\item $T_1: m_1 = x, l_1 = a_1$, $\mu_1= [b^{-1}, y^{-1}]$,
\item $T_2:  m_2 = y, l_2 = b_1a_1b^{-1},$ $\mu_2 = [x^{-1}, b_1]$,
\item $T_3:  m_3 = x, l_3 = a_2$, $\mu_3 = [b_2^{-1}, y^{-1}]$,
\item $T_4: m_4 = y, l_4 = b_2a_2b_2^{-1}$, $\mu_4 = [x^{-1}, b_2]$,
\item $\mu_{\Sigma_2} = [x, y]$.
\end{itemize}

The loops $a_1, b_1, a_2, b_2$ lie on the genus 2 surface and form a standard set of generators; the relation $[a_1, b_1][a_2, b_2] = 1$ holds.\\
\end{proposition}

The final building block is obtained by applying Luttinger surgeries along Lagrangian tori to the product of two genus 2 surfaces $\Sigma_2\times \Sigma_2$ equipped with the product symplectic form. Let $x_1, y_1, x_2, y_2$ be the generators of $\pi_1(\Sigma\times \{x\})$, and $a_1, b_1, a_2, b_2$ be the generators of $\pi_1(\{x\}\times \Sigma_2)$. 

\begin{proposition} \cite[Lemma 16]{[ABBKP]}, \cite[Section 4]{[FPS]}.{\label{Proposition 16}} There exists a minimal symplectic 4-manifold $Z$ with $c_1^2(Z) = 8$ and $\chi_h(Z) = 1$ that contains eight homologically essential Lagrangian tori $\{S_1, S_2, S_3, S_4, S_5, S_6, S_7, S_8\}$ (each $S_i$ has a geometrically dual torus $S_i^d$ so that all other intersections are zero). The fundamental group $\pi_1(Z - (S_1 \cup \cdots \cup S_8))$ is generated by $x_1, y_1, x_2, y_2$ and $a_1, b_1, a_2, b_2$, and the meridians and Lagrangian push-offs are given by
\begin{itemize}
\item $S_1: \mu_1 = [b_1^{-1}, y_1^{-1}], m_1 = x_1, l_1 = a_1$, 
\item  $S_2: \mu_2 = [x_1^{-1}, b_1], m_2 = y_1, l_2 = b_1 a_1 b_1^{-1}$,
\item  $S_3: \mu_3 = [b_2^{-1}, y_1^{-1}], m_3 = x_1, l_3 = a_2$,
\item $S_4: \mu_4 = [x_1^{-1}, b_2], m_4 = y_1, l_4 = b_2a_2b_2^{-1}$,
\item $S_5: \mu_5 = [b_1a_1^{-1}b_1^{-1}, y_2^{-1}], m_5 = x_2, l_5 = b_1^{-1}$, 
\item  $S_6: \mu_6 = [x_2^{-1}, b_1a_1b_1^{-1}], m_6 = y_2, l_6 = b_1 a_1 b_1^{-1} a_1^{-1}b_1^{-1}$,
\item $S_7: \mu_7 = [b_2 a_2^{-1} b_2^{-1}, y_2^{-1}], m_7 = x_2, l_7 = b_2^{-1}$,
\item $S_8: \mu_8 = [x_2^{-1}, b_2a_2b_2^{-1}], m_8 = y_2, l_8 = b_2 a_2 b_2^{-1} a_2^{-1} b_2^{-1}$.

\end{itemize}
\end{proposition}

\subsection{Luttinger and torus surgeries}{\label{Section 4.5}}
Let $T$ be a Lagrangian torus in a symplectic 4-manifold $X$. By the Darboux-Weinstein Theorem \cite{[McS]}, there exists a parametrization $T^2\times D^2 \cong N_T \subset X$ of a tubular neighborhood $N_T$ of $T$, for which the image of $T^2 \times \{d\}$ ($d\in D^2$) is Lagrangian. Let $d\in D^2 - \{0\}$, the \emph{Lagrangian push off} or \emph{Lagrangian framing} of $T$ is the push off
\begin{center}
 $F_d: T \rightarrow T^2 \times \{d\} \subset X - T$
 \end{center}
 determined by $d$. The smooth isotopy class of $F_d: T \rightarrow X - T$ depends only on the symplectic structure of $X$ around a neighborhood of $T$. 
 Let $\gamma$ be an embedded curve in $T$. The \emph{Lagrangian push off} of $\gamma$ is the isotopy class of the image $F_d(\gamma)$. A \emph{meridian} of the torus $T$ is a curve in the isotopy class of $\{t\} \times \partial D^2 \subset \partial (N_T)$, and it will be denoted by $\mu_T$.\\

Let $\alpha$ and $\beta$ be the generators of $\pi_1(T)$. Let $m_T = F_d(\alpha)$ and $l_T = F_d(\beta)$ for $d\in \partial D^2$ be the push offs of the loops in $\partial N_T = T^3$. These are loops homologous in $N_T$ to $\alpha$ and $\beta$ respectively. The loops $\{m, l, \mu_T\}$ generate $H_1(\partial(N_T); \mathbb{Z}) = \mathbb{Z}^3$. We choose a basepoint $t$ laying on $\partial(N_T)$, and set $m_T, l_T, \mu_T \in \pi_1(\partial(N_T), t) = \mathbb{Z}^3$.\\

The manifold obtained from $X$ by performing a $(p, q, n)$-torus surgery on $T$ along $\gamma$ is defined as

\begin{center}
$X_{T, \gamma}(p, q, n) := (X - N_T) \cup_{\phi} (T^2 \times D^2)$,
\end{center}

where the gluing map $\phi: T^2\times \partial D^2 \rightarrow \partial (X - N_T)$ satisfies \begin{center}$\phi_*([\partial D^2]) = p[m_T] + q[l_T] + n[\mu_T]$ in $H_1(\partial (X - N_T)); \Z)$.\end{center}

For $n = 1$, the procedure described above is known as \emph{Luttinger surgery on $T$ along $\gamma \subset T$}, and $X_{T, \gamma}(p, q, 1)$ admits a symplectic structure \cite{[Lu], [ADK]}.\\

If the base point of $X$ is not on the boundary of the tubular neighborhood $N_T$ of $T$, then the based loops $\mu$ and $\gamma$ are to be joined by the same path in $X - T$. The core torus of the surgery $S^1\times S^1 \times \{0\} \subset X_{T, \gamma}(p, q, 1)$ will be denoted by $T_{(p, q, 1)}$. Regarding the fundamental group of the manifolds $X_{T, \gamma}(p, q, n)$ constructed in this paper, we the following considerations suffice. We will fix generating curves $\alpha, \beta$ on the torus $T$, and express the curve as $\gamma = \alpha^a \beta^b$ in $\pi_1(T)$ for $a, b\in \{0, 1\} $ such that $a \neq b$.

\begin{lemma}{\label{Lemma 13}} (cf. \cite[Lemma 4]{[BK2]}). The manifold obtained by applying a $(p, q, n)$-torus surgery to $X$ on the torus $T$ along the curve $\gamma = a\alpha + b\beta$ has fundamental group given by\begin{center}
$\pi_1(X_{T, \gamma}(p, q, n)) \cong \pi_1(X - T)/N(\mu_T^nm_T^{ap}l_T^{bq})$,
\end{center}
where $N(\mu_T^nm_T^pl_T^q)$ denotes the normal subgroup generated by $\mu_T^nm_T^pl_T^q$.
\end{lemma}

Every torus surgery that is employed in this paper will require for only one of the two integers $p$ or $q$ to be nonzero. From now on, the data needed to specify a torus surgery will be encoded in the following terminology. A $\pm p/n$-torus surgery on the torus $T$ along the curve $m_T$ implies we are taking $q = 0$; analogously, a $\pm q/n$-torus surgery on the torus $T$ along the curve $l_T$ implies that we are setting $p = 0$.

\subsection{Two examples}{\label{Section 4.7}} Before we start filling vast regions, the results and building blocks from the previous sections are used to build hands-on examples. 

\begin{example} {\label{Example 17}} With the goal of constructing irreducible 4-manifolds with $c_1^2 = 3$, and $\chi_h = 1$, consider $T^2\times T^2$ and $T^2\times S^2$ equipped with the product symplectic forms.
We begin by finding symplectic surfaces of genus 2 inside of each of these manifolds. For $t_1, t_2\in T^2$, take   $T^2 \times \{t_2\} \cup \{t_1\} \times T^2 \subset T^2\times T^2$. By symplectically resolving the double point, we obtain a symplectic surface of genus 2 and self intersection $([T^2\times \{t_2\}] + [\{t_1\}\times T^2])^2 = 2$.  Blow up at the two intersection points, and obtain a symplectic surface of genus 2 and self intersection zero $\Sigma_2 \subset T^2\times T^2 \# 2 \overline{\mathbb{CP}^2}$ \cite[Building block 5.7]{[Gom]}.\\

Analogously, we obtain a symplectic surface of genus 2 that has self intersection zero $\Sigma_2' \subset T^2\times S^2 \# 3 \overline{\mathbb{CP}^2}$, as follows (cf. \cite{[RF1]}, \cite[Section 3]{[AP]}). Consider the embedded torus $T'$ that represents the homology class $2[T^2\times \{s_1\}]$, where $s_1\in S^2$. Take the union of this torus with $T^2\times \{s_1\} \cup \{t_1\} \times S^2$ for a given $t_1\in T^2$. Blow up one of the double points, and symplectically resolve \cite{[Gom]} the remaining pair of double points in order to get a genus 2 surface of self intersection 2. Blow up at these intersection points to obtain the desired $\Sigma_2' \subset T^2\times S^2 \# 3 \overline{\mathbb{CP}^2}$.\\ 

Since both $\Sigma_2$ and $\Sigma_2'$ have trivial self intersection, there exist diffeomorphisms $\nu(\Sigma_2) \rightarrow \Sigma_2 \times D^2$ and $\nu(\Sigma_2')\rightarrow \Sigma_2' \times D^2$, where $\nu(\Sigma_2), \nu(\Sigma_2')$ are tubular neighborhoods of the surfaces inside $T^4\# 2 \overline{\mathbb{CP}^2}$ and $T^2\times S^2 \# 3 \overline{\mathbb{CP}^2}$ respectively. Build the generalized fiber sum along the genus 2 surfaces

\begin{center}
$Z:= (T^2\times T^2 \# 2 \overline{\mathbb{CP}^2} ) \#_{\Sigma_2 = \Sigma_2'} (T^2\times S^2 \# 3 \overline{\mathbb{CP}^2})$.\\
\end{center}

A well-known result of Gompf \cite{[Gom]} (cf. \cite{[MW]}) implies that $Z$ is a symplectic manifold, and its  characteristic numbers  are given by $c_1^2(Z) = 3$, and $\chi_h(Z) = 1$. A result of Usher \cite[Theorem 1.1]{[Ush]} implies that $Z$ is minimal. In order to conclude its irreducibility using the criteria of Hamilton-Kotschick \cite[Corollary 1]{[HKo]}, we proceed to compute its fundamental group using the Seifert-van Kampen theorem.\\

The standard presentations of the groups involved in the cut-and-paste construction of $Z$ are

\begin{center}
$\pi_1(T^4) = \left\langle x, y, a, b | [x, y] = [a, b] = [x, a] = [x, b] = [y, a] = [y, b] = 1 \right\rangle$,
\end{center}

\begin{center}
$\pi_1(T^2\times S^2) = \left\langle \alpha, \beta | [\alpha, \beta] = 1\right\rangle$,
\end{center}

\begin{center}
 $\pi_1(\Sigma_2) = \left\langle a_1, b_1, a_2, b_2 | [a_1, b_1][a_2, b_2] = 1\right\rangle$, and
\end{center}

\begin{center}
$\pi_1(\Sigma_2') = \left\langle a_1', b_1', a_2', b_2' | [a_1', b_1'][a_2', b_2'] = 1\right\rangle$.
\end{center}

Let the inclusion $\Sigma_2 \hookrightarrow T^2\times T^2 \# 2 \overline{\mathbb{CP}^2}$ induce the homomorphism that maps the generators as follows:

\begin{center}
$a_1\mapsto x$, $b_1\mapsto y$\\
$a_2\mapsto a$, $b_2 \mapsto b$.
\end{center}

Similarly, let the homomorphism induced by the inclusion, $\Sigma_2' \hookrightarrow T^2\times S^2 \# 3 \overline{\mathbb{CP}^2}$ map the generators of the fundamental groups by

\begin{center}
$a_1'\mapsto \alpha$, $b_1'\mapsto \beta^2$\\
$a_2'\mapsto \alpha^{-1}$, $b_2' \mapsto \beta^{-2}$
\end{center}

(cf. \cite[Section 3]{[AP]}). Take the diffeomorphism $\phi: \Sigma_2 \times \partial D^2 \rightarrow \Sigma_2' \times \partial D^2$ that induces the homomorphism on fundamental groups that identifies the generators as follows:

\begin{center}
$a_1\mapsto a_1'$, $b_1\mapsto b_1'$\\
$a_2\mapsto a_2'$, $b_2 \mapsto b_2'$.
\end{center}


Furthermore, in each case the meridians of the surface are homotopically trivial in the complement. Indeed, an exceptional sphere introduced by a blow up intersects the surface transversally in one point. Thus, the meridians $\mu_{\Sigma_2}, \mu_{\Sigma_2'}$ are nullhomotopic in $T^4 \# 2\overline{\mathbb{CP}^2} - \nu(\Sigma_2)$ and $T^2\times S^2 \# 3\overline{\mathbb{CP}^2} - \nu(\Sigma_2')$, respectively.\\

Using the Seifert-van Kampen theorem, we conclude that the group $\pi_1(Z)$ is generated by the elements $x, y, a, b, \alpha, \beta$, and that they satisfy the relations

\begin{center}
$[x, y] = [a, b] = [x, a] = [x, b] = [y, a] = [y, b] = [\alpha, \beta] = 1$, $x = \alpha, y = \beta^2, a = \alpha^{-1}, b = \beta^{-2}$.  
\end{center}

Thus, $\pi_1(Z) = \left\langle \alpha, \beta | [\alpha, \beta^2] = 1\right\rangle \cong \Z \oplus \Z_2$.\\

Since $\Z\oplus \Z_2$ is a residually finite group, the symplectic 4-manifold $Z$ is irreducible \cite[Corollary 1]{[HKo]}. In order to produce irreducible and minimal manifolds with $\pi \in \{\Z_2, \Z_{p_1}\oplus \Z_2\}$, we proceed as follows. The work of Taubes indicates that the Seiberg-Witten invariant of the symplectic manifold $Z$ is nontrivial \cite{[T1]}. Moreover, it contains two pairs of Lagrangian tori $\{T_1, T_1^{d}\}, \{T_2, T_2^{d}\}$ such that $T_i$ intersects $T_i^{d}$  ($i = 1, 2$) transversally at one point, and any other intersection is empty. These tori are inherited from the $T^2\times T^2$ building block in the construction of $Z$. By Proposition \ref{Proposition 14}, the Lagrangian pushoff and meridian and of the torus $T_1$ are 

\begin{center}
$T_1: m_1 = \alpha, \mu_{T_1} = 1$.
\end{center}

Perform a $(p_1/n)$-torus surgery on $T_1$ along $m_1$ to produce an infinite set of pairwise nondiffeomorphic manifolds
\begin{center}
$\{Z_{p_1, 2, n} : p_1, n \in \N\}$.
\end{center}

By the Seifert-van Kampen theorem (cf. Lemma \ref{Lemma 13}), \begin{center} $\pi_1(Z_{p_1, 2, n}) \cong \left\langle \alpha, \beta | \alpha^{p_1} = \beta^{2} = 1^n = [\alpha, \beta^2] \right\rangle \cong \Z_{p_1}\oplus \Z_2$.\end{center}

The characteristic numbers are invariant under torus surgeries. This yields $c_1^2(Z_{p_1, 2, n}) = c_1^2(Z) = 3$, and $\chi_h(Z_{p_1, 2, n}) = \chi_h(Z) = 1$ for all $p_1, n \in \N$. As it was mentioned in Section \ref{Section 4.5}, the $(p_1/1)$-torus surgery is known as Luttinger surgery, and the manifolds $\{Z_{p_1, 2, 1}\}$ are symplectic \cite{[Lu], [ADK]}. As in \cite{[FPS], [FS], [ABP], [NFS]} the diffeomorphism classes of the manifolds $\{Z_{p_1, 2, n}\}$ are distinguished by the Seiberg-Witten invariants \cite[Theorem 5.3]{[FS]}, which are calculated according to the surgery coefficient $n$ of the torus surgery by the Morgan-Mrowka-Szab\'o formula \cite{[MMS]}, using the formulation given in \cite[Theorem 3.4]{[SZs]} (please, see proof of Lemma \ref{Lemma 19} for more details). We have produced then an infinite set $\{Z_{p_1, 2, n}: p_1, n \in \N\}$ of pairwise nondiffeomorphic nonsymplectic minimal 4-manifolds.\\

To finish the example, we claim that the manifold 
\begin{center}
$\mathbb{CP}^2 \# 6 \overline{\mathbb{CP}^2} \# L_{2}$
\end{center}
has the $\infty$-property \ref{Definition 1} for $p_1 = 1$.
\end{example} 

Let $\{Z_{1, 2, n}: n \in \N\}$ be the infinite set of pairwise nondiffeomorphic minimal manifolds with fundamental group of order two that was previously constructed. We proceed to pin down the homeomorphism type of these manifolds. Minding the fundamental group, we have computed $\pi_1(Z_{1, 2, n}) \cong \Z_{2}$. The characteristic numbers are $e(Z_{1, 2, n}) = e(Z) = 9, \sigma(Z_{1, 2, n}) = \sigma(Z) = - 5$, since both the Euler characteristic and the signature are invariant under torus surgeries. Let $\widetilde{Z_{1, 2, n}}$ be the universal cover of $Z_{{1}, 2, n}$. Its signature is $\sigma(\widetilde{Z_{p_1, 1, n}})  = - 10$; by Rokhlin's Theorem \cite{[Rok]}, $\widetilde{Z_{p_1, 1, n}}$ is nonspin. This implies that, for arbitrary $n\in \N$, the manifold $Z_{1, 2, n}$ has $\omega_2$-type (I). By Theorem \ref{Theorem 6}, $Z_{1, 2, n}$ is homeomorphic to $\mathbb{CP}^2 \# 6 \overline{\mathbb{CP}^2} \# L_{2}$. The claim follows from the existence of the infinite set $\{Z_{1, 2, n}: n\in \N\}$ of pairwise nondiffeomorphic manifolds that was constructed before.\\


\begin{example}{\label{Example 18}} In this example, irreducible 4-manifolds with $c_1^2 = 1$, and $\chi_h = 1$ are constructed by modifying our choice of gluing map in the construction of the symplectic sum. We begin by unveiling a surface of genus 2 and self intersection zero symplectically embedded in $T^2\times S^2 \# 4 \overline{\mathbb{CP}^2}$ that intersects each exceptional sphere at a point. Consider the product of a 2-torus and a 2-sphere $T^2\times S^2$ equipped with the product symplectic form. Take the union $T^2\times \{s_1\} \cup \{t_1\} \times S^2 \cup T^2 \times \{s_2\}$ for $t_1\in T^2$ and $s_1, s_2 \in S^2$. This wedge of symplectic surfaces has two double points, which come form the intersection of $\{t_1\}\times S^2$ with the two tori. Resolving symplectically \cite{[Gom]} results in a surface of genus 2 and self intersection 4. Blow up four times, and obtain a surface of genus 2 and self intersection zero $\Sigma_2$ that is symplectically embedded in $T^2\times S^2 \# 4 \overline{\mathbb{CP}^2}$.

We build the generalized fiber sum

\begin{center}
$Z_{\phi}: = T^2\times S^2 \# 3 \overline{\mathbb{CP}^2} \#_{\Sigma_2' = \Sigma_2}  T^2\times S^2 \# 4 \overline{\mathbb{CP}^2}$,
\end{center}

where $\Sigma_2' \subset T^2\times S^2 \# 3 \overline{\mathbb{CP}^2}$ are taken as in Example \ref{Example 17}, and the glueing map is the diffeomorphism $\phi: \partial \nu(\Sigma_2') \rightarrow \partial \nu(\Sigma_2)$ of the tubular neighborhoods of the genus 2 surfaces. The same arguments used in Example \ref{Example 17} imply that $Z_{\phi}$ is a minimal symplectic 4-manifold with characteristic numbers $c_1^2(Z_{\phi}) = 1$ and $\chi_h(Z_{\phi}) = 1$.

Minding the fundamental group, we have the following computations. Let the groups involved have the presentations

\begin{center}
$\pi_1(T^2\times S^2 \# 3 \overline{\mathbb{CP}^2}) = \left\langle \alpha, \beta | [\alpha, \beta] = 1\right\rangle$,
\end{center}

\begin{center}
$\pi_1(T^2\times S^2 \# 4 \overline{\mathbb{CP}^2}) = \left\langle x, y | [x, y] = 1\right\rangle$,
\end{center}

\begin{center}
$\pi_1(\Sigma_2') = \left\langle a_1', b_1', a_2', b_2' | [a_1', b_1'][a_2', b_2'] = 1\right\rangle$,
\end{center}

\begin{center}
 $\pi_1(\Sigma_2) = \left\langle a_1, b_1, a_2, b_2 | [a_1, b_1][a_2, b_2] = 1\right\rangle$, and
\end{center}

Let the homomorphism induced by the inclusion, $\Sigma_2' \hookrightarrow T^2\times S^2 \# 3 \overline{\mathbb{CP}^2}$, map the generators of the fundamental groups as

\begin{center}
$a_1'\mapsto \alpha$, $b_1'\mapsto \beta^2$\\
$a_2'\mapsto \alpha^{-1}$, $b_2' \mapsto \beta^{-2}$.
\end{center}

Similarly, let the homomorphism induced by the inclusion, $\Sigma_2 \hookrightarrow T^2\times S^2 \# 4 \overline{\mathbb{CP}^2}$, map the generators of the fundamental groups as

\begin{center}
$a_1\mapsto x$, $b_1\mapsto y$\\
$a_2\mapsto x^{-1}$, $b_2\mapsto y^{-1}$.
\end{center}

Different choices of gluing maps in the construction of $Z_{\phi}$ yield different symplectic manifolds with abelian fundamental groups of rank at most two. For instance, choose as gluing map the diffeomorphism $\phi: \Sigma_2 \times \partial D^2 \rightarrow \Sigma_2' \times \partial D^2$ that induces the isomorphism on fundamental groups that identifies the generators as follows:

\begin{center}
$a_1'\mapsto a_1$, $b_1' \mapsto b_1$\\
$a_2'\mapsto a_2$, $b_2' \mapsto b_2$.
\end{center}

We conclude that $\pi_1(Z_{\phi}) = \left\langle x, y| y^2 = 1, [ x, y] = 1\right\rangle \cong \Z \oplus \Z_2$ using the Seifert-van Kampen theorem.\\

We proceed to construct an irreducible 4-manifold with fundamental group of order two. Inside the surface $\Sigma_2$, let $A$ and $B$ be parallel curves to $\alpha$ and $\beta$ respectively. Let $\phi_1: \Sigma_2 \times \partial D^2 \rightarrow \Sigma_2' \times \partial D^2$ be the diffeomorphism obtained from composing three Dehn twists $D_A D_B D_A$ with $\phi$ (cf. \cite[Section 4.2]{[BK2]}). The composition induces the isomorphism on fundamental groups that assigns the generators as follows:

\begin{center}
$a_1'\mapsto b_1^{-1}$, $b_1' \mapsto b_1a_1b_1^{-1}$\\
$a_2'\mapsto a_2$, $b_2' \mapsto b_2$.
\end{center}

By the Seifert-van Kampen theorem, we have that $\pi_1(Z_{\phi_1})$ is generated by $x, y, \alpha, \beta$ and the following relations hold

\begin{center}
$[\alpha, \beta] = [x, y] = 1$, $y^2 = 1$, $\alpha = y^{-1}$, $\beta = y x y^{-1}$, $\alpha^{-1} = x^{-1}$, $\beta^{-1} = y^{-1}$.
\end{center}

This implies $\alpha = \beta^{-1}$, $x = y^{-1}$, and $x^2 = 1$. Thus, $Z_{\phi_1}$ is an irreducible symplectic 4-manifold with cyclic fundamental group of order two.
\end{example}

\subsection{Technical tools.}{\label{Section 4.8}} We now present two useful lemmas that allow us to extend results on simply connected manifolds to manifolds with abelian fundamental groups.

\begin{lemma}{\label{Lemma 19}} Let $G\in \{ \{1\}, \Z_p, \Z_p\oplus \Z_q, \Z\oplus \Z_q, \Z, \Z\oplus \Z\}$. Let $X$ be an irreducible symplectic simply connected 4-manifold, which contains a homologically essential Lagrangian torus $T$ of self intersection zero such that $\pi_1(X - T) = \{1\}$. Then, there exists an irreducible symplectic 4-manifold $X^G$ with
\begin{itemize}
\item $\pi_1(X^G) = G$,
\item $c_1^2(X^G) = c_1^2(X)$, and
\item $\chi_h(X^G) = \chi_h(X)$.
\end{itemize}
If $G\in \{ \{1\}, \Z_p, \Z_q\oplus \Z_q, \Z \}$, where $q$ is an odd prime number, then $X^G$ has the $\infty$-property.
Moreover, suppose $X$ contains a second homologically essential Lagrangian torus $T'$ of self intersection zero, disjoint from $T$ and $\pi_1(X - T \cup T') = \{1\}$. Then, there exists an infinite family $\{X^G_n : n \in \N\}$ of pairwise nondiffeomorphic 4-manifolds with $\omega_2$-type (I) that consists of infinitely many irreducible symplectic and infinitely many minimal nonsymplectic members, such that
\begin{itemize}
\item $\pi_1(X^G_n) = G$,
\item $c_1^2(X^G_n) = c_1^2(X)$, and
\item $\chi_h(X^G_n) = \chi_h(X)$
\end{itemize}
for all $n\in \N$. If $G \in \{\{1\}, \Z_p, \Z_q \oplus \Z_q, \Z\}$, then $X^G$ has the $\infty^2$-property.
\end{lemma}

\begin{proof} Consider the 4-torus $T^2\times T^2$ equipped with the product symplectic form. Perturb the symplectic form on $X$ so that $T$ becomes symplectic \cite[Lemma 1.6]{[Gom]}. Build the generalized fiber sum
\begin{center}
$Z:= X \#_{T = T'} T^2\times T^2$,
\end{center}

where $T':= \{x\} \times T^2 \subset T^4$. Using the notation in Proposition \ref{Proposition 15}, by the Seifert-van Kampen theorem we have $\pi_1(Z) \cong \Z x \oplus \Z y$; the elements $x$ and $y$ each generate an infinite cyclic factor of $\Z^2$. As done in Example \ref{Example 17}, an iterated usage of \cite{[Gom], [Ush], [HKo]} allows us to conclude that $Z$ is an irreducible symplectic 4-manifold with characteristic numbers $c_1^2(Z) = c_1^2(X)$ and $\chi_h(Z) = \chi_h(X)$. 

The two Lagrangian tori $T_1$ and $T_2$ are contained inside the manifold $Z$ and they are available for surgery. Using the notation of Proposition \ref{Proposition 15}, apply a $(-p/1)$- and a $(-q/n)$-torus surgery on $T_1$ and $T_2$ along the curves $x$ and $y$ respectively, to produce a set of 4-manifolds\begin{center} $\{Z_{p, q, n}: p, q, n \in \N\}$. \end{center} Lemma \ref{Lemma 13} yields $\pi_1(Z_{p, q, n}) \cong \left\langle x, y | x^p = 1, y^q = 1^n, [x, y] = 1 \right\rangle$.

The characteristic numbers are invariant under torus surgeries. Therefore, we have that $c_1^2(Z_{p, q, n}) = c_1^2(Z) = c_1^2(X)$, and $\chi_h(Z_{p, q, n}) = \chi_h(Z) = \chi_h(X)$ for all $p, q, n \in \N$. The manifolds $\{Z_{p, q, 1} : p, q\in \N\}$ are symplectic since the case $n = 1$  corresponds to Luttinger surgeries \cite{[Lu], [ADK]}. 

The infinite set of pairwise nondiffeomorphic minimal nonsymplectic 4-manifolds $\{Z_{p, q, n} : p, q, n  \in \N\}$ can be constructed using the techniques of \cite{[FPS], [FS], [ABP], [NFS]}. The set is reverse-engineered from the symplectic 4-manifold $Z$ in the sense of \cite[Section 2]{[FPS]}, \cite[Section 2]{[NFS]} given that the work of Taubes indicates that the Seiberg-Witten invariant of $Z$ is nontrivial \cite{[T1]}. To verify that the needed hypothesis are satisfied, we use an argument in \cite[Remark p. 343]{[BK3]}. For clarity, we discuss the case $G = \Z$; the other choices of fundamental group follow from a similar argument. The group $\pi_1(\partial \nu(T_2)) \cong \pi_1(T^3)$ is generated by the loops $m_2, l_2$, and $\mu_2$. In the group $\pi_1(Z - \nu(T_2))$, the identities $m_2 = y, l_2 = 1$, and $\mu_2 = 1$ hold by Proposition \ref{Proposition 14}. Construct a manifold with infinite cyclic fundamental group
\begin{center}
$Z_{1, 0} = (Z - \nu(T_2)) \cup_{\phi_{Z_{1, 0}}} T^2\times D^2$
\end{center}

using gluing map $\phi_{Z_{1, 0}}$ that is defined as follows. Let $T^2\times D^2 = S^1\times S^1\times D^2$, and define the curves $\alpha := S^1\times \{1\} \times \{1\}, \beta:= \{1\}\times S^1 \times \{1\}$, and $\mu_{Z_{1, 0}}:= \{(1, 1)\} \times \partial D^2$ that are contained in $S^1\times S^1\times D^2$. Let $\phi_{Z_{1, 0}}$ be the gluing map defined by the identification of curves $\alpha \mapsto l_2, \beta \mapsto \mu_2$, and $\mu_{Z_{1, 0}}\mapsto m_2^{-1}$. 

The core torus $\widetilde{T}:= T^2\times \{0\} \subset T^2\times D^2$ of the surgery is nullhomologous. Notice that we have $Z - \nu(T_2) = Z_{1, 0} - \nu(\widetilde{T})$. Indeed, the meridian $\mu_{\widetilde{T}}$ is nontrivial in $Z_{1, 0} - \nu(\widetilde{T})$ due to our identification. The torus surgery that was applied kills a generator of the first homology group, as well as two generators of the second homology group. These generators correspond to the class of the torus to which we apply the surgery, and the class of its dual. Moreover, the curve $\beta$ is nullhomotopic in $Z - \nu(T_2) = Z_{1, 0} - \nu(\widetilde{T})$, since it was identified with $\mu_2$. Thus, the curve $\beta$ is nullhomologous in the complement. The manifold $Z$ can be recovered from $Z_{1, 0}$ by applying a $1/0$-torus surgery on $\widetilde{T}$ along the curve $\beta$. Applying a $n/1$-torus surgery on $\widetilde{T}$ along $\beta$ for $n\in \N$ produces a 4-manifold $Z_{1, 0, n}$ with infinite cyclic fundamental group. Another description of this manifold is as the product of applying surgery to $Z$ on $T_2$ along $m_2$ with the identification $\mu_2^n m_2^{-1} \mapsto y^{-1}$, i.e., a $1/n$-torus surgery.

The hypothesis of reverse-engineering of \cite{[FPS],[FS], [NFS]} are satisfied. The diffeomorphism classes are distinguished by the Seiberg-Witten invariants \cite[Theorem 5.3]{[FS]}, which are calculated with the Morgan-Mrowka-Szab\'o formula \cite{[MMS]} using the formulation given in \cite[Theorem 3.4]{[SZs]}. We have then produced an infinite set $\{Z_{1, 0, n} :  n  \in \N\}$ of pairwise nondiffeomorphic minimal nonsymplectic 4-manifolds \cite[Theorem 5.3]{[FS]}, \cite[Theorem 1]{[NFS]} with infinite cyclic fundamental group. Notice that we are abusing notation regarding the infinite set. Since we do not know if the manifold $Z$ has only one Seiberg-Witten basic class up to sign, we can not conclude that the manifolds $Z_{1, 0, n}$ are pairwise nondiffeomorphic (cf. \cite[Corollary 1]{[FPS]}). However, we do know that $Z$ has a basic class  \cite{[T1]}, and \cite[Corollary 2]{[FPS]} implies that among the manifolds $\{Z_{1, 0, n}: n\in \N\}$ there are infinitely many pairwise nondiffeomorphic. We have kept the same notation for them.

The presence of a homologically essential torus $T' \subset Z$ with the homomorphism $\pi_1(T')\rightarrow \pi_1(Z_{p, q, 1}) \cong G$ induced by the inclusion allows us to perform Fintushel and Stern's Knot surgery \cite{[FS3]}, \cite[Theorem 4]{[RF]}, \cite[Lecture 3]{[FS6]} (see \cite[Corollary 24]{[BPa3]} as well), and produce infinitely many pairwise nondiffeomorphic 4-manifolds $\{X^G_n: n\in \N\}$. Given that $T'$ is homologically essential, by \cite[Lemma 1.6]{[Gom]} one can perturb the symplectic form so that $T'$ becomes symplectic. A result of W. Thurston \cite{[WT]} implies that for fibered knots, the manifolds constructed in this way are symplectic \cite[Remark 10.3.5]{[GS]}. The presence of a symplectic structure allows us to conclude on irreducibility using the results in \cite{[HKo]}. 

In order to conclude on the $\infty$- and $\infty^2$-properties, we examine the homeomorphism types of the constructed manifolds. The hypothesis on the symplectic torus of self intersection zero contained in the minimal symplectic manifold $X$ implies that $b_2^+(X) \geq 3$. Thus, the stability conditions of Theorem \ref{Theorem 8} and of Theorem \ref{Theorem 9} are satisfied. Let $b^+:= b_2^+(X)$ and $b^-:= b_2^-(X)$.  By Theorem \ref{Theorem 7}, the manifolds in the set $\{Z_{p, 1, n}: n \in \N\}$ are homeomorphic to $b^+ \mathbb{CP}^2 \# b^- \overline{\mathbb{CP}^2} \# L_{p}$. Assuming $p = q$ to be an odd prime number, Theorem \ref{Theorem 8} implies that the manifolds in the set $\{Z_{q, q, n}: n\in \N\}$ are homeomorphic to 
$b^+ \mathbb{CP}^2 \# b^- \overline{\mathbb{CP}^2} \# L_{q, q}$. Theorem \ref{Theorem 9} implies that every manifold with infinite fundamental group constructed above is homeomorphic to $(b^+ + 1) \mathbb{CP}^2 \# (b^- + 1) \overline{\mathbb{CP}^2} \# S^3\times S^1$.
\end{proof}

The previous result can be extended when the manifold $X$ in the hypothesis of Lemma \ref{Lemma 19} contains a submanifold of higher genus. The genus two case is exemplified in the following lemma.

\begin{lemma}{\label{Lemma 20}}  Let $G\in \{ \{1\}, \Z_p, \Z_p\oplus \Z_q, \Z\oplus \Z_q, \Z, \Z\oplus \Z\}$. Let $X$ be an irreducible symplectic simply connected 4-manifold, which contains a symplectic surface $\Sigma$ of genus 2 and self intersection zero  such that $\pi_1(X - \Sigma) = \{1\}$. Then, there exists an infinite family $\{X^G_n : n\in \N\}$ of pairwise nondiffeomorphic 4-manifolds with $\omega_2$-type (I) that consists of infinitely many irreducible symplectic and infinitely many minimal nonsymplectic members, such that 
\begin{itemize}
\item $\pi_1(X^G_n) = G$,
\item $c_1^2(X^G_n) = c_1^2(X) + 8$, and
\item $\chi_h(X^G_n) = \chi_h(X) + 1$
\end{itemize}
for all $n\in \N$. If $G\in \{ \{1\}, \Z_p, \Z_q\oplus \Z_q, \Z \}$, where $q$ is an odd prime number, then $X^G$ has the $\infty^2$-property.
\end{lemma}

\begin{proof} Consider the product $T^2\times \Sigma_2$ equipped with the product symplectic form. Build the generalized fiber sum
\begin{center}
$Z:= X \#_{\Sigma = \Sigma_2} T^2\times \Sigma_2$
\end{center}

along the corresponding genus 2 surfaces. By Proposition \ref{Proposition 15}, the loops $a_1, b_1, a_2, b_2$ lie on the genus 2 surface. The Seifert-van Kampen theorem implies that the fundamental group is $\pi_1(Z) \cong \Z x \oplus \Z y$, since we assumed that $\pi_1(X - \Sigma) = \{1\}$. Arguing as before, we conclude that $Z$ is an irreducible symplectic manifold with characteristic numbers $c_1^2(Z) = c_1^2(X) + 8$ and $\chi_h(Z) = \chi_h(X) + 1$.

The symplectic manifold $Z$ contains 4 homologically  Lagrangian tori available for surgery, which are contained in the $T^2\times \Sigma_2$ block of the symplectic sum. Perform $-1/1$ Luttinger surgeries on $T_1$ and $T_2$ along $x$ and $y$ respectively. This results in an irreducible symplectic simply connected 4-manifold $Y$ that contains two pairs of homologically essential Lagrangian tori $T_3$ and $T_4$, such that $Y - (T_3 \cup T_4)$ is simply connected by Proposition \ref{Proposition 15}. The lemma now follows from Lemma \ref{Lemma 19}.
\end{proof}



\subsection{Telescoping triples.}{\label{Section 4.9}} The basic building blocks used  to systematically populate geographical regions of irreducible symplectic 4-manifolds are presented in this section.

\begin{definition} (\cite[Definition 2]{[ABBKP]}). An ordered triple $(X, T_1, T_2)$ consisting of a symplectic 4-manifold $X$ and two disjointly 
embedded Lagrangian tori $T_1$ and $T_2$ is called a telescoping triple if 
\begin{enumerate}
\item The tori $T_1$ and $T_2$ span a 2-dimensional subspace of $H_2(X; \R)$.
\item $\pi_1(X) \cong \Z^2$ and the inclusion induces an isomorphism\begin{center} $\pi_1(X - (T_1 \cup T_2)) \rightarrow \pi_1(X)$.\end{center} In particular, the meridians of the tori are trivial in $\pi_1(X - (T_1 \cup T_2))$.
\item The image of the homomorphism induced by the corresponding inclusion $\pi_1(T_1) \rightarrow \pi_1(X)$ is a summand $\Z \subset \pi_1(X)$.
\item The homomorphism induced by inclusion $\pi_1(T_2) \rightarrow \pi_1(X)$ is an isomorphism.
\end{enumerate}
The telescoping triple is called minimal if $X$ itself is minimal.
\end{definition}

The order of the tori in the definition is relevant. The meridians $\mu_{T_1}$, $\mu_{T_2}$ in $\pi_1(X - (T_1 \cup T_2))$ are trivial, and the fundamental groups involved are abelian. The push off of an oriented loop $\gamma \subset T_i$ into $X - (T_1 \cup T_2)$ with 
respect to any (Lagrangian) framing of the normal bundle of $T_i$ represents a well defined element of 
$\pi_1(X - (T_1 \cup T_2) )$, which is independent of the choices of framing and base-point.\\

The first condition says that Lagrangian tori $T_1$ and $T_2$ are linearly independent in $H_2(X; \R)$. The symplectic form on $X$ can be slightly perturbed so that one of the $T_i$ remains Lagrangian while the other becomes 
symplectic \cite[Lemma 1.6]{[Gom]}.

Out of two telescoping triples, one is able to produce another one as follows.

\begin{proposition} \cite[Proposition 3]{[ABBKP]}.{\label{Proposition 22}} Let $(X, T_1, T_2)$ and $(X', T_1', T_2')$ be two telescoping triples. Then for an appropriate gluing map the triple

\begin{center}
$(X \#_{T_2 = T_1'} X', T_1, T_2')$
\end{center}
is again a telescoping triple. If $X$ and $X'$ are minimal symplectic 4-manifolds, then the resulting telescoping triple is minimal. The characteristic numbers of $X \#_{T_2 = T_1'} X'$ are  \begin{center} $c_1^2(X) + c_1^2(X')$ and $\chi_h(X) + \chi_h(X')$.\end{center}
\end{proposition}

Minimal telescoping triples are irreducible by \cite[Corollary 1]{[HKo]}. The existing telescoping triples are gathered in the following result, which was proven in \cite[Section 5]{[ABBKP]}, \cite{[To]}.

\begin{theorem}{\label{Theorem 23}} Existence of telescoping triples.
\begin{itemize} 
\item There exists a minimal telescoping triple $(A, T_1, T_2)$ satisfying $c_1^2(A) = 7$, $\chi_h(A) = 1$.
\item For each $g\geq 0$, there exists a minimal telescoping triple $(B_g, T_1, T_2)$ satisfying $c_1^2(B_g) = 6 + 8g$, $\chi(B_g) = 1 + g$.
\item There exists a minimal telescoping triple $(C, T_1, T_2)$ satisfying $c_1^2(C) = 5$, $\chi_h(C) = 1$.
\item There exists a minimal telescoping triple $(D, T_1, T_2)$ satisfying $c_1^2(D) = 4$, $\chi_h(D) = 1$.
\item There exists a minimal telescoping triple $(F, T_1, T_2)$ satisfying $c_1^2(F) = 2$, $\chi_h(F) = 1$.\\
\end{itemize}
\end{theorem}

\begin{remark}{\label{Remark 24}} \emph{The telescoping triples are obtained in the construction process of minimal symplectic 4-manifolds homeomorphic to $\mathbb{CP}^2 \# k\overline{\mathbb{CP}^2}$ \cite{[FPS], [A1], [AP], [BK2], [BK3], [AP1]}. In \cite{[NFS]}, Fintushel and Stern introduced a procedure to unveil exotic smooth structures on these manifolds by performing surgeries on nullhomologous tori that are contained inside rational surfaces (cf. \cite{[FS]}). Minimal 4-manifolds with abelian fundamental group of small rank lying on the line $\chi_h = 1$ can be constructed using the techniques of \cite{[NFS]}. This requires only a small variation on the choices of surgeries on the pinwheel structures on \cite[Section 7, Section 8]{[NFS]}.}
\end{remark}

An immediate consequence of Theorem \ref{Theorem 23} is the study of the geography and botany of irreducible symplectic 4-manifolds for various fundamental groups (cf. \cite{[FM], [LO], [OKo], [HK88], [HK90], [HK93], [Wa], [Gom], [BK1], [OS], [IS], [BK4],[BK3], [AP1], [ABBKP], [AP], [To], [Y]}). In particular, the following result extends \cite[Theorem 1]{[To]}.

\begin{corollary}{\label{Corollary 25}}  Let $G \in \{\Z_p, \Z_q\oplus \Z_p, \Z, \Z \oplus \Z_q, \Z \oplus \Z\}$, and let $n\geq 1$ and $m\geq 1$. For each of the following pairs of
integers
\begin{enumerate}
\item $(c, \chi) = (7n, n)$,
\item $(c, \chi) = (5n, n)$,
\item $(c, \chi) = (4n, n)$,
\item $(c, \chi) = (2n, n)$,
\item $(c, \chi) = ((6 + 8g)n, (1 + g)n)$ (for $g\geq 0$),
\item $(c, \chi) = (7n + (6 + 8g)m, n + (1 + g)m)$,
\item $(c, \chi) = (7n + 5m, n + m)$,
\item $(c, \chi) = (7n + 4m, n + m)$,
\item $(c, \chi) = (7n + 2m, n + m)$,
\item $(c, \chi) = ((6 + 8g)n + 5m,(1 + g)n + m)$ (for $g\geq 0$),
\item $(c, \chi) = ((6 + 8g)n + 4m,(1 + g)n + m)$ (for $g\geq 0$),
\item $(c, \chi) = ((6 + 8g)n + 2m,(1 + g)n + m)$ (for $g\geq 0$),
\item $(c, \chi) = (5n + 4m , n + m)$,
\item $(c, \chi) = (5n + 2m , n + m)$, and
\item $(c, \chi) = (4n + 2m , n + m)$,
\end{enumerate}
there exists an irreducible symplectic 4-manifold $X^G$ with $\omega_2$-type (I) and
\begin{center} 
$\pi_1(X^G) = G$ and $(c_1^2(X^G), \chi_h(X^G)) = (c, \chi)$.
\end{center}
Moreover, if $G = \Z_p$, then $X^G$ has the $\infty$-property.
Assume $q$ to be an odd prime number, and let $\chi > 1$. If $G \in \{\Z_q\oplus \Z_q, \Z\}$, then $X^G$ has the $\infty$-property.
\end{corollary}


\begin{remark}{\label{Remark 26}} Symplectic universal covers. \emph{It is natural to ask whether the universal covers of the symplectic manifolds constructed are standard or exotic (we thank Paul Kirk for bringing it to our attention). If they were standard, then the action of the fundamental group would be exotic. Exotic smooth actions were studied in \cite{[FSS]}. The universal covers in our constructions admit a symplectic structure; thus, they are exotic. This phenomena has previously been observed in \cite{[FST]}. Indeed, let $\pi: \widetilde{X}\rightarrow X$ be the cover of a symplectic manifold $X$ of $\omega_2$-type (I) with finite cyclic fundamental group, and denote its symplectic structure by $\omega$. By Freedman's theorem \cite{[F]}, $\widetilde{X}$ is homeomorphic to}

\begin{center}
$(p(b_2^+ + 1) - 1) \mathbb{CP}^2 \# (p(b_2^- + 1) - 1) \overline{\mathbb{CP}^2}$.
\end{center}

\emph{The cover $\widetilde{X}$ admits a symplectic structure, since the pullback $\pi^{\ast}\omega$ is a symplectic form on $\widetilde{X}$.  By Taubes' result \cite{[T1]}, the Seiberg Witten invariant of $\widetilde{X}$ is nontrivial. Since $(p(b_2^+ + 1) - 1) \mathbb{CP}^2 \# (p(b_2^- + 1) - 1) \overline{\mathbb{CP}^2}$ has vanishing Seiberg-Witten invariants \cite[Theorem 2.4.6]{[GS]}, $\widetilde{X}$ is not diffeomorphic to it.}

\end{remark}

\section{Filling in geographical regions.}{\label{Section 5}}

We prove Theorem \ref{Theorem 2} in this section. We will first construct irreducible 4-manifolds with negative signature for every possible lattice point on the line $\chi_h = 2$, and use these manifolds to construct more manifolds on the lines $\chi_h > 2$ using Proposition \ref{Proposition 16}. The following bound is a guide to which 4-manifolds are to be constructed.

\begin{lemma}{\label{Lemma 27}} Let $X$ be a minimal symplectic 4-manifold. Then
\begin{center} $b_2^-(X) \leq 10\chi_h(X) - 1$. \end{center}
\end{lemma}

\begin{proof} By the work of Taubes \cite{[T0], [T2]} (see also \cite[Corollary 4.9]{[Ko]}), a minimal symplectic 4-manifold $X$ must satisfy $c_1^2(X) \geq 0$ (\cite[Corollary 13.1.12]{[OS1]}). A straight forward computation shows that the restriction implies \begin{center}$b_2^-(X) \leq 5 b_2^+(X) + 4 = 10 (\chi_h(X)) - 1$.\end{center}
\end{proof}

There are other restrictions for 4-manifolds with $b_1\leq 1$ \cite[Lemma 2.6]{[S2]}; Lemma \ref{Lemma 27} suffices for our purposes.

\subsection{Irreducible manifolds on the line $\chi_h = 2$}{\label{Section 5.1}} With the purpose of proving Theorem \ref{Theorem 2}, we begin to systematically populate every lattice point on the line $\chi_h = n$ for every  integer $n\geq 2$ that corresponds to manifolds with negative signature, with an irreducible  4-manifold. For the line $\chi_h = 2$, we employ the following result.

\begin{proposition}{\label{Proposition 28}} Let $b^- \in \{4, 5, \ldots, 17, 18, 19\}$. There exists a minimal nonspin symplectic simply connected 4-manifold $X_{3, b^-}$ with second Betti number given by $b_2^+(X) = 3$ and $b_2^-(X) = b^-$, which contains two homologically essential Lagrangian tori $T_1$ and $T_2$ such that \begin{center} $\pi_1(X_{3, b^-} - (T_1 \cup T_2)) = \{1\}$.\end{center}
\end{proposition}

The manifolds of Proposition \ref{Proposition 28} unveil exotic smooth structures on $3\mathbb{CP}^2 \# b_2^- \overline{\mathbb{CP}^2}$. This enterprise has already been done previously \cite{[Gom], [SZs1], [S1], [BPa], [BPa1], [BPa2], [JPa1], [JPa2], [SS2], [BK1], [A1], [BK3], [BK4], [AP1], [ABP], [ABBKP], [AP], [PKH]}. What concerns us for our purposes is the existence of the tori $T_1$ and $T_2$ with the required traits (cf. \cite[Theorem 18]{[BK3]}) in order to be able to apply Lemma \ref{Lemma 19}.

\begin{table}[ht] {\label{Table 1}}
\caption{Minimal symplectic 4-manifolds of Proposition \ref{Proposition 28}.}
\centering
\begin{tabular}{c | c | c | c | c }
\hline\hline
$b_2^-$ & $(c_1^2, \chi_h, \sigma)$ & Symplectic sum/manifold. & $\#$ of Luttinger surgeries. & Reference. \\
\hline\hline
$4$   & $(15, 2, -1)$  & $X_{1, 2} \#_{\Sigma_2} T^2\times \Sigma_2$ & $two$ & \cite[Section 9]{[AP]}  \\
$5$    & $(14, 2, -2)$  & $X_{3, 5}$ & $none$ & \cite[Theorem 18]{[BK3]} \\
$6$    & $(13, 2, -3)$ & $X_{1, 2} \#_{\Sigma_2} (T^4\# 2\overline{\mathbb{CP}^2})$ & $none$ & \cite[Section 9]{[AP]}\\
$7$    & $(12, 2, -4)$ &$X_{3, 7}$ & $none$ & \cite[Corollary 15]{[BK3]} \\
$8$   & $(11, 2, -5)$ & $X_{1, 5}\#_{\Sigma_2} (T^4\# \overline{\mathbb{CP}^2})$ & $none$ & \cite[Sections 3 and 4]{[AP]}.\\
$9$  & $(10, 2, -6)$ & $X_{1, 5} \#_{\Sigma_2} (T^4 \# 2\overline{\mathbb{CP}^2})$ & $none$ & \cite[Corollary 16]{[BK3]}.\\

$10$  &$(9, 2, -7)$ & $X_{1, 6} \#_{\Sigma_2} (T^4 \# 2 \overline{\mathbb{CP}^2})$ & $none$ & \cite[Lemma 15]{[AP]}  \\
$11$  & $(8, 2, -8)$ & $E(1) \#_{T^2} S$ & $six$ & \cite[Lemma 16]{[ABBKP]}  \\
$12$   & $(7, 2, -9)$ & $E(1) \#_{T^2} A$ & $none$ & Theorem \ref{Theorem 23} \\
$13$  & $(6, 2, -10)$ & $E(1) \#_{T^2} B$ & $none$ & Theorem \ref{Theorem 23}  \\
$14$  & $(5, 2, -11)$ & $\mathbb{CP}^2 \# 12 \overline{\mathbb{CP}^2} \#_{\Sigma_2} T^2\times \Sigma_2$ & $two$ & \cite[Building Block 5.6]{[Gom]}  \\
$15$  & $(4, 2, -12)$ & $\mathbb{CP}^2 \# 13 \overline{\mathbb{CP}^2} \#_{\Sigma_2} T^2\times \Sigma_2$ & $two$ & \cite[Building Block 5.6]{[Gom]}  \\
$16$ & $(3, 2, -13)$ & $\mathbb{CP}^2 \# 12 \overline{\mathbb{CP}^2} \#_{\Sigma_2} (T^4\# 2 \overline{\mathbb{CP}^2})$ & $none$ & \cite[Building Blocks 5.6 and 5.6] {[Gom]}  \\
$17$ & $(2, 2, -14)$ & $\mathbb{CP}^2 \# 13 \overline{\mathbb{CP}^2} \#_{\Sigma_2} (T^4\# 2 \overline{\mathbb{CP}^2})$ & $none$ & \cite[Building Blocks 5.6 and 5.7]{[Gom]}  \\
$18$ & $(1, 2, -15)$ & $S_{1, 1}$ & $none$ & \cite[Example 5.4]{[Gom]} \\
$19$ & $(0, 2,  -16)$ & $E(1)\#_{T^2} T^4 \#_{T^2} E(1)$ & $none$ & \cite{[SZs]}  \\


\hline
\end{tabular}
\label{table:masspr}
\end{table}

\begin{proof} Table \ref{Table 1} provides a guideline of where the manifolds of Proposition \ref{Proposition 28} are built or obtained from, as well as the number of Luttinger surgeries required in their construction. We will prove the proposition for one case, given that the argument for the other manifolds is analogous. We proceed to construct an irreducible symplectic simply connected 4-manifold  $X_{3, 15}$ with \begin{center} $(c_1^2(X_{3, 15}), \chi_h(X_{3, 15}), \sigma(X_{3, 15}))= (4, 2, -12)$ \end{center} that contains the two tori as in the statement of the proposition. 
Take a quartic curve inside $\mathbb{CP}^2$ with a single transverse double point. By blowing up at this singular point, a surface of genus 2 and self intersection 12 is obtained in $\mathbb{CP}^2 \# \overline{\mathbb{CP}^2}$. Proceed to blow up at each of these different points. The proper transform $F$of the twelve blow ups is a complex curve of genus 2 and self intersection zero, which is contained in $\mathbb{CP}^2 \# 13 \overline{\mathbb{CP}^2}$ (cf. \cite[Building block 5.6]{[Gom]}). From the construction is clear that the surface $F$ intersects every exceptional sphere from each of the thirteen blow ups. Moreover, $\pi_1(\mathbb{CP}^2 \# 13 \overline{\mathbb{CP}^2} - F) = \{1\}$, since the exceptional sphere introduced during the thirteenth blow up provides a nullhomotopy for the meridian of $F$ once it is removed from the manifold.

Build the generalized fiber sum
\begin{center}
$Z_{4, 2}:= (\mathbb{CP}^2 \# 13 \overline{\mathbb{CP}^2}) \#_{F = \Sigma_2} (T^2\times \Sigma_2)$.
\end{center}

Again, the manifold $Z_{4, 2}$ is symplectic by \cite{[Gom]}, it is minimal by \cite[Theorem 1.1]{[Ush]}. Its characteristic numbers are readily computed to be $c_1^2(Z_{4, 2}) = c_1^2(\mathbb{CP}^2 \# 13 \overline{\mathbb{CP}^2}) + c_1^2(T^2\times \Sigma_2) + 8 = 4$, and $\chi_h(Z_{4, 2}) = \chi_h(\mathbb{CP}^2 \# 13 \overline{\mathbb{CP}^2}) + \chi_h(T^2\times \Sigma_2) + 1 = 2$.

One concludes $\pi_1(Z_{4, 2}) \cong \mathbb{Z} x \oplus \mathbb{Z} y$ from Proposition \ref{Proposition 15} and the Seifert-van Kampen theorem. Since the fundamental group is residually finite, $Z_{4, 2}$ is irreducible by \cite[Corollary 1]{[HKo]}.
There are eight pairs of geometrically dual Lagrangian tori contained in $Z_{4, 2}$ that come from the $T^2\times \Sigma_2$ block in the construction of the symplectic sum.  We apply torus surgeries to $Z_{4, 2}$ following Proposition \ref{Proposition 15} and the notation there. Perform $-1/1$-torus surgeries on $T_3$ and $T_4$ along $m_3 = x$ and $m_4 = y$ respectively to obtain a simply connected irreducible symplectic 4-manifold $X_{3, 15}$ \cite{[Lu], [ADK]}, Proposition \ref{Proposition 15}.
Both tori $T_1$ and $T_2$ have each a geometrically dual torus inside $X_{3, 15}$. Thus, $T_1$ and $T_2$ are homologically essential Lagrangian tori. From Proposition \ref{Proposition 15} one concludes that their meridians are trivial in the complement, and one has $\pi_1(X_{3, 15} - (T_1 \cup T_2)) = \{1\}$.
\end{proof}

\subsection{Region of negative signature: proof of Theorem \ref{Theorem 2}}{\label{Section 5.2}} We will fill out the region of the geography of irreducible symplectic 4-manifolds whose fundamental group is among the choices $\{\Z_p, \Z_p \oplus \Z_q, \Z\oplus \Z_q, \Z, \Z\oplus \Z\}$ by constructing irreducible manifolds with the given choices of fundamental groups realizing all pairs of integers

\begin{center}
$(c_1^2, \chi_h)$ when $0\leq c_1^2 \leq 8\chi_h - 1$.
\end{center}

Note that under the chosen coordinates, a 4-manifold with $c^2_1 = 8\chi + k$ has signature $k$. We proceed to prove the main result Theorem \ref{Theorem 2}.

\begin{proof} The homeomorphism types are pinned down by the work of Hambleton-Kreck, Hambleton-Teichner presented in Section \ref{Section 4.1}, and Section \ref{Section 4.2}. A worked out argument on these matters was given in Example \ref{Example 17}. The line $\chi_h = 1$ was populated and its botany studied in Corollary \ref{Corollary 25}, and Example \ref{Example 17} (see Remark \ref{Remark 24} also). The line $\chi_h = 2$ is populated as a consequence of Proposition \ref{Proposition 28}, and Lemma \ref{Lemma 19}. 
Assume $n\geq 3$. The region
 \begin{center}$\chi_h = n$,\end{center} \begin{center}$8n - 16 \leq c_1^2 \leq 8n - 1$\end{center} in statement of Theorem \ref{Theorem 2} is populated by an iteration on the construction of generalized fiber sums of the manifolds of Proposition \ref{Proposition 28} with $n$ copies of the manifold $Z$ of Proposition \ref{Proposition 16} along tori. Notice that by \cite[Lemma 1.6]{[Gom]}, the symplectic form on these manifolds can be perturbed so that the Lagrangian tori become symplectic. The Seifert-van Kampen theorem implies that the fundamental group is trivial. The claim now follows from Lemma \ref{Lemma 19}.
The region \begin{center}$\chi_h = n$,\end{center} \begin{center} $0 \leq c_1^2 < 8n - 16$\end{center} is populated by forming the generalized fiber sum of the manifolds in Proposition \ref{Proposition 28} and the telescoping triples of Theorem \ref{Theorem 23} along tori. Lemma \ref{Lemma 27} serves as guidance as of which telescoping triple is to be used according to its signature. The torus involved in the gluing of each telescoping triple is the torus $T_2$; the Seifert-van Kampen theorem implies that the fundamental group of the symplectic sum is trivial. The reader will notice that there are several ways choices on the building blocks used on the constructions of the generalized fiber sums; various points can be populated alternatively by using Lemma \ref{Lemma 20}. The claim now follows from an iterated application of \cite{[Gom], [Ush], [HKo]}, and Lemma \ref{Lemma 19} as before.
\end{proof}

\subsection{Regions of nonnegative signature}{\label{Section 5.3}} The tools of Section \ref{Section 4.8} can also be used to extend results on simply connected 4-manifolds with nonnegative signature \cite{[Sti], [BK], [ABBKP], [JPa], [AP3]} to other fundamental groups. We give a sample result in this section, first recalling the following theorem.



\begin{theorem} (Akhmedov -  B.D. Park \cite{[AP3]}). For each of the following pairs $(c, \chi) \in \Z \oplus \Z$
\begin{itemize}
\item $(c, \chi) = (200, 25)$,
\item $(c, \chi) = (201, 25)$,
\item $(c, \chi) = (204, 24)$,
\item $(c, \chi) = (219, 27)$, and
\item $(c, \chi) = (212, 26)$,
\end{itemize}
there exists a minimal symplectic simply connected 4-manifold $X$ with characteristic numbers $c_1^2(X) = c$ and $\chi_h(X) = \chi$. Moreover, $X$ contains two disjoint homologically essential Lagrangian tori, $T_1, T_2$, each of self-intersection zero, and such that $\pi_1(X - T_1 \cup T_2) = \pi_1(X - T_i) = \{1\}$ for $i \in \{1, 2\}$.
\end{theorem}

The previous theorem and Lemma \ref{Lemma 19} yield the following result.

\begin{corollary}{\label{Corollary 29}} Let $G\in \{\{1\}, \Z_p, \Z_p\oplus \Z_q, \Z, \Z\oplus \Z_q, \Z \oplus \Z\}$, and assume $m, r, s, t$ to be odd positive integers that satisfy $m\geq 49, r\geq 47, s\geq 53$, and $t\geq 51$. For each of the following pairs of integers 

\begin{enumerate}
\item $(c, \chi) = (4m + 4, \frac{1}{2}(m + 1))$,
\item $(c, \chi) = (4m+ 5, \frac{1}{2}(m + 1))$,
\item $(c, \chi) = (4r + 6, \frac{1}{2}(r + 1))$ ,
\item $(c, \chi) = (4s + 7, \frac{1}{2}(s + 1))$, and
\item $(c, \chi) = (4t + 8, \frac{1}{2}(t + 1))$,
\end{enumerate}

there exists a 4-manifold $X^G$ with $\omega_2$-type (I) and
\begin{center}
$\pi_1(X^G) = G$ and $(c_1^2(X^G), \chi_h(X^G)) = (c, \chi)$.
\end{center}
If $G \in \{\{1\}, \Z, \Z_p, \Z_q\oplus \Z_q\}$ for $q$ an odd prime number, then the homeomorphism type of $X^G$ has the $\infty^2$-property.
\end{corollary}

\section{Acknowledgments}

The author thanks I. Baykur for suggesting the problem, and M. Marcolli and P. Kirk for their encouragement and support. We thank A. Akhmedov, S. Baldridge, D. Calegari, R. Fintushel, I. Hambleton, M. Kreck, B. D. Park, R. J. Stern, and the referees for comments on an earlier version of this paper. We thank the math department at Caltech, the math department at the State University of Florida and, especially, we are thankful to the Max-Planck Institut f\"ur Mathematik - Bonn for their warm hospitality and excellent working conditions. This work was supported by an IMPRS
scholarship from the Max-Planck Society.

\end{document}